\documentclass[12pt,a4paper,reqno]{amsart}
\usepackage{graphicx}
\usepackage[linktocpage=true,colorlinks,citecolor=blue,linkcolor=blue,urlcolor=blue]{hyperref}
\usepackage{amssymb}
\usepackage{cite}
\usepackage{bbm}
\usepackage{amsmath}
\usepackage{latexsym}
\usepackage{amscd}
\usepackage{tikz}
\usepackage{amsthm}
\usepackage{mathrsfs}
\usepackage{url}
\usepackage[utf8]{inputenc}
\usepackage[english]{babel}
\usepackage{amsfonts}
\usepackage{mathtools}
\usepackage{comment}
\usepackage[autostyle]{csquotes}
\usepackage[colorinlistoftodos]{todonotes}

\numberwithin{equation}{section}
\def\R{\mathbb R}
\def\Z{\mathbb Z}

\def\P{\mathbb P}

\def\E{\mathbb E}

\def\r{\right}
\def\l{\left }

\def\ee{\varepsilon}

\def\gcd{\operatorname{gcd}}
\def\rad{\operatorname{rad}}

\DeclarePairedDelimiter\ceil{\lceil}{\rceil}

\def\le {\leqslant}
\def\ge {\geqslant}

\newtheorem{theorem}{Theorem}[section]
\newtheorem{lemma}[theorem]{Lemma}
\newtheorem{proposition}[theorem]{Proposition}

\newtheorem{corollary}[theorem]{Corollary}

\theoremstyle{remark}
\newtheorem{remark}[theorem]{Remark}

\theoremstyle{definition}
\newtheorem{definition}[theorem]{Definition}
\newtheorem{question}[theorem]{Question}
\numberwithin{equation}{section}

\theoremstyle{remark}

\begin{document}
\title[Typical multiplicative functions over short moving intervals]{Partial sums of typical multiplicative functions over short moving intervals}
\author{Mayank Pandey}
\address{Department of Mathematics, Princeton University, Princeton, NJ, USA}
\email{mayankpandey9973@gmail.com}
\author{Victor Y. Wang}
\address{Department of Mathematics, Princeton University, Princeton, NJ, USA}
\address{Courant Institute, New York University, New York, NY, USA}
\email{vywang@alum.mit.edu}
\author{Max Wenqiang Xu}
\address{Department of Mathematics, Stanford University, Stanford, CA, USA}
\email{maxxu@stanford.edu}


\begin{abstract}
We prove that the $k$-th positive integer moment of partial sums of Steinhaus random multiplicative functions over the interval $(x, x+H]$ matches the corresponding Gaussian moment, as long as $H\ll x/(\log x)^{2k^2+2+o(1)}$ and $H$ tends to infinity with $x$. We show that properly normalized partial sums of typical multiplicative functions arising from realizations of random multiplicative functions have Gaussian limiting distribution in short moving intervals $(x, x+H]$ with $H\ll X/(\log X)^{W(X)}$ tending to infinity with $X$, where $x$ is uniformly chosen from $\{1,2,\dots, X\}$, and $W(X)$ tends to infinity with $X$ arbitrarily slowly. This makes some initial progress on a recent question of Harper.
\end{abstract}

\maketitle

\section{Introduction}
We are interested in the partial sums behavior of a family of completely multiplicative functions $f$ supported on moving short intervals.
Formally, for positive integers $X$, let $[X] := \{1,2,\dots, X\}$ and
\[\mathcal{F}_X :=
\{ f: [X]\to \{|z|=1\},\; \text{$f$ is completely multiplicative}  \}. \]
For $f\in \mathcal{F}_X$, the function values $f(n)$ for all $n\le X$ are uniquely determined by $(f(p))_{p\le X}$.
The Steinhaus random multiplicative function is defined by selecting $f(p)$ uniformly at random from the complex unit circle and defining $f(n)$ completely multiplicatively. One may view $\mathcal{F}_X$ as the family of all Steinhaus random multiplicative functions. 

Let $H$ be another positive integer.
We are interested in for a typical $f\in \mathcal{F}_{X+H}$, whether the random partial sums
\begin{equation}
\label{EQN:notation-A_H(f,x)-for-normalized-partial-sum}
    A_H(f, x) := \frac{1}{\sqrt{H}}\sum_{x<n \le x+H} f(n),
\end{equation}
where $x$ is uniformly chosen from $[X]$, behave like a complex standard Gaussian.
In this note, we provide a positive answer (Theorem~\ref{thm: main}) when $H\ll_A X/\log^AX$ holds for all $A > 0$.
As we explain in \S\ref{sec: end}, the answer is negative for 
$H\gg X\exp(-(\log\log X)^{1/2 - \ee})$, 
but the question remains open between these two thresholds.

We formalize the question by explaining how to measure the elements in $\mathcal{F}_X$.
Via complete multiplicativity of $f\in \mathcal{F}_X$, define on $\mathcal{F}_X$ the product measure
\[\nu_X:= \prod_{p\le X}\mu_p,\]
where for any given prime $p$, we let $\mu_p$ denote the uniform distribution on the set $\{f(p)\} = \{|z|=1\}$.
For example, $\nu_X(\mathcal{F}_X) =1$.

\begin{question}[{Harper\cite[open question (iv)]{Harper-moving}}]\label{question}
What is the distribution of the normalized random sum defined in \eqref{EQN:notation-A_H(f,x)-for-normalized-partial-sum} (for most $f$) as $x$ is uniformly chosen from $[X]$?
\end{question}

\subsection{Main results}
In this note, we make some progress on Question~\ref{question}.
We use the notation $\xrightarrow{d}$ to denote convergence in distribution. 
\begin{theorem}\label{thm: main}
Let integer $X$ be large and $W(X)$ tend to infinity arbitrarily slowly as $X$ tends to infinity. Let $H: =  H(X) \ll X(\log X)^{-W(X)}$ and $H\to +\infty$ as $X \to +\infty$.
Then, for almost all $f\in \mathcal{F}_{X+H}$, as $X\to +\infty$,
\begin{equation}
\label{EQN:short-interval-statistic-converges-in-distribution-to-Gaussian}
\frac{1}{\sqrt{H}} \sum_{x<n \le x+H} f(n)  \xrightarrow{d} \mathcal{CN}(0,1),
\end{equation}
where $x$ is chosen uniformly from $[X]$. 
\end{theorem}
Here \enquote{almost all} means the total measure of such $f$ is $1-o_{X\to +\infty}(1)$ under $\nu_{X+H}$.\footnote{More precisely, there exist nonempty measurable sets $\mathcal{G}_{X, H}\subseteq \mathcal{F}_{X+H}$ of measure $1-o_{X\to +\infty}(1)$ (under $\nu_{X+H}$) such that for every sequence of functions $f_X\in \mathcal{G}_{X, H}$ ($X\ge 1$), the random variable on the left-hand side of \eqref{EQN:short-interval-statistic-converges-in-distribution-to-Gaussian} (with $f = f_X$) converges in distribution to $\mathcal{CN}(0,1)$ as $X\to +\infty$.}
Also, $\mathcal{CN}(0,1)$ denotes the standard complex normal distribution;
a standard complex normal random variable $Z$ (with mean $0$ and variance $1$) can be written as $Z=X+iY$, where $X$ and $Y$ are independent real normal random variables with mean $0$ and variance $1/2$.
Recall that a real normal random variable $W$ with mean $0$ and variance $\sigma^{2}$ satisfies
	\[\P(W\le t) = \frac{1}{\sigma\sqrt{2\pi}} \int_{-\infty}^{t}e^{-\frac{x^{2}}{2\sigma^{2}}}dx . \]


To prove Theorem~\ref{thm: main}, we establish moment statistics in several situations.
We first show that the integer moments of random multiplicative functions $f$ supported on suitable short intervals match the corresponding Gaussian moments.
We write $\E_f$ to mean ``average over $f\in \mathcal{F}_X$ with respect to $\nu_X$'' (where $\mathcal{F}_X$ is always clear from context).


\begin{theorem}\label{thm: high moments}
Let $x,H,k\ge 1$ be integers.
Let $f\in \mathcal{F}_{x+H}$.
Let $E(k) = 2k^2+2$.
Then
\[
\E_f \biggl|\frac{1}{\sqrt{H}}\sum_{x<n\le x+H}f(n)\biggr|^{2k}
= k! + O_k\biggl(H^{-1}
+ \frac{H^{1/2}}{\max(x, H)^{1/2}}
+ \frac{H\cdot (\log{x} + \log{H})^{E(k)}}{\max(x, H)}\biggr),
\]
with an implied constant depending only on $k$.
\end{theorem}

Notice that $k!$ is the $2k$-th moment of the standard complex Gaussian distribution.
Given an integer $k\ge 1$, let $E'(k)$ be the smallest real number $r\ge 0$ such that for every $\ee>0$, we have $\E_f |A_H(f,x)|^{2k} \to k!$ whenever
\begin{equation*}
x\to +\infty
\qquad\textnormal{and}\qquad
(\log{x})^\ee \le H \le x/(\log x)^{r+\ee}.
\end{equation*}
Theorem~\ref{thm: high moments} shows that $E'(k)\le E(k)$.\footnote{After writing the paper, the authors learned that for $H\le x/\exp(C_k\log{x}/\log\log{x})$, the Diophantine statement underlying Theorem~\ref{thm: high moments} has essentially appeared before in the literature (see \cite[proof of Theorem~34]{Bourgain2014}).
However, we handle a more delicate range of the form $H\le x/(\log x)^{Ck^2}$.}
The paper \cite{CS} studies the case $k=2$, showing in particular that $E'(2)\le 1$.
In the case that $f$ is supported on $\{1,2\dots, x\}$, the $2k$-th moments for general $k$ were studied in \cite{batyrev1995manin, Breteche, breteche2001toric, GS2001,   HeapLind,  Harperhigh, HarperHelson} and it is known that the moments there do not match Gaussian moments:
for instance, by \cite[Theorem~1.1]{Harperhigh}, there exists some constant $c>0$ such that for all positive integers $k\le c \frac{\log x}{\log \log x}$ (assuming $x$ is large),
\begin{equation}\label{eqn: high}
    \E_f \biggl|\frac{1}{\sqrt{x}} \sum_{n\le x} f(n)\biggr|^{2k} = e^{-k^{2}\log(k\log(2k)) + O(k^{2})}  (\log x)^{(k-1)^{2}} .
\end{equation}
While \eqref{eqn: high} is quite uniform over $k$, it is unclear how uniform in $k$ one could make our Theorem~\ref{thm: high moments}.
(See Remark~\ref{RMK:k-uniformity} for some discussion of the $k$-aspect in our work.)

\begin{remark}
The powers of $\log x$ above are significant.
For instance, Theorem~\ref{thm: high moments} in the range $H\gg x$ follows directly from \eqref{eqn: high}, since $(k-1)^2\le E(k)$.
One may also wonder how far our bound $E'(k)\le E(k)$ is from the truth.
Based on a circle method heuristic for \eqref{eqn: counting} along the lines of \cite[Conjecture~2]{hooley1986some}, with a Hardy-Littlewood contribution on the order of $\frac{H^{2k}}{Hx^{k-1}} (\log{x})^{(k-1)^2}$, and an additional contribution of roughly $k!H^k$ from trivial solutions,
it is plausible that one could improve the right-hand side in Theorem~\ref{thm: high moments} to $k! + O_k(\frac{H^{k-1}}{x^{k-1}}(\log{x})^{(k-1)^2})$ for $H\in [x^{1-\delta}, x]$.
If true, this would suggest that $E'(k)\le k-1$ and we believe this might be the true order. 
For a discussion of how one might improve on Theorem~\ref{thm: high moments}, see the beginning of \S\ref{sec: end}.
\end{remark}

By orthogonality, Theorem~\ref{thm: high moments} is a statement about the Diophantine point count
\begin{equation}\label{eqn: counting}
    \#\{(n_1, n_2,\dots, n_{2k})\in (x, x+H]^{2k}: n_1n_2\cdots n_k = n_{k+1}n_{k+2}\cdots n_{2k}\}.
\end{equation}
The circle method, or modern versions thereof such as \cite{duke1993bounds,heath1996new},
might lead to an asymptotic for \eqref{eqn: counting} uniformly over $H\in [x^{1-\delta}, x]$
for $k=2$,  unconditionally (cf.~\cite[Theorem~6]{heath1996new}),
or for $k=3$, conditionally on standard number-theoretic hypotheses (cf.~\cite{wang2021_HLH_vs_RMT}).
Alternatively, ``multiplicative'' harmonic analysis along the lines of \cite{Breteche,HarperHelson,HeapLind}
may in fact lead to an unconditional asymptotic over $H\in [x^{1-\delta}, x]$ for all $k$, with many main terms involving different powers of $\log{x}, \log{H}$.
Nonetheless, for all $k$, we obtain an unconditional asymptotic for \eqref{eqn: counting} uniformly over $H\ll x/(\log x)^{Ck^{2}}$,
by replacing the complicated ``off-diagonal'' contribution to \eqref{eqn: counting} with a \emph{larger but simpler} quantity;
see \S\ref{SEC:high-moments-proof} for details.

\begin{remark}
An analog of \eqref{eqn: counting} for polynomial values $P(n_i)$ is studied in \cite{KSX, wangxu},
and a similar flavor counting question to \eqref{eqn: counting} is studied in \cite{FGM} using the decoupling method.
\end{remark}

After Theorem~\ref{thm: high moments},
our next step towards Theorem~\ref{thm: main} is to establish concentration estimates for the moments of the random sums \eqref{EQN:notation-A_H(f,x)-for-normalized-partial-sum}.
We write $\E_x$ to denote ``expectation over $x$ uniformly chosen from $[X]$'' (where $X$ is always clear from context).

\begin{theorem}\label{thm: Steinhaus}
Let $X,k\ge 1$ be integers with $X$ large.
Suppose that 
$H:=H(X)\to +\infty $ as $X\to +\infty$.
There exists a large absolute constant $A>0$ such that the following holds as long as $H\ll X(\log X)^{-C_k}$ with $C_k = A k^{A k^{A k}}$.
Let $f\in \mathcal{F}_{X+H}$;
then
\begin{equation}\label{eqn: 2k}
     \E_f \biggl( \E_x \biggl|\frac{1}{\sqrt{H}}\sum_{x<n \le x+H} f(n)\biggr|^{2k} -k!\biggr)^{\!2}
     = o_{X\to +\infty}(1).
\end{equation}
Furthermore, for any fixed positive integer $\ell < k$, we have
\begin{equation}\label{eqn: kl}
    \E_f \biggl|\E_x \biggl(\frac{1}{\sqrt{H}}\sum_{x<n \le x+H} f(n)\biggr)^{\!k} \biggl(\frac{1}{\sqrt{H}}\sum_{x<n \le x+H} \overline{f(n)}\biggr)^{\!\ell}\, \biggr|^{2}
    = o_{X\to +\infty}(1).
\end{equation}
\end{theorem}

We prove Theorem~\ref{thm: high moments} in \S\ref{SEC:high-moments-proof}, and then we prove Theorem~\ref{thm: Steinhaus} in \S\ref{SEC:proof-concentration-theorem}.

\begin{proof}[Proof of Theorem~\ref{thm: main}, assuming Theorem~\ref{thm: Steinhaus}]
We use the notation $A_H(f,x)$ from \eqref{EQN:notation-A_H(f,x)-for-normalized-partial-sum}.
By Markov's inequality, Theorem~\ref{thm: Steinhaus} implies that there exists a set of the form
\begin{equation*}
\begin{split}
\mathcal{G}_{X, H} := \{ f\in \mathcal{F}_{X+H}:
\E_x |A_H(f, x)|^{2k} - k! &= o_{X\to +\infty}(1)~\text{for all $k\le V(X)$}, \\
\E_x \Bigl[A_H(f, x)^k \, \overline{A_H(f,x)^\ell}\Bigr] &= o_{X\to +\infty}(1)~\text{for all distinct $k,\ell\le V(X)$}
\}
\end{split}
\end{equation*}
for some $V(X) \to +\infty$ (making a choice of $V(X)$ based on $W(X)$) such that
\[\nu_{X+H}(\mathcal{G}_{X, H}) = 1-o_{X\to +\infty}(1).\]
Since the distribution $\mathcal{CN}(0,1)$ is uniquely determined by its moments (see e.g.~\cite[Theorem~30.1 and Example~30.1]{Bill}),
Theorem~\ref{thm: main} follows from the method of moments \cite[Chapter~5, Theorem~8.6]{Gut} (applied to sequences of random variables $A_H(f,x)$ indexed by $f\in \mathcal{G}_{X, H}$ as $X\to +\infty$).
\end{proof}

We believe results similar to our theorems above should also hold in the (extended) Rademacher case, though we do not pursue that case in this paper.

\subsection{Notation}
For any two functions $f, g: \R\to \R$, we write $f\ll g, g\gg f, g= \Omega(f)$ or $f = O(g)$ if there exists a positive constant $C$ such that $|f|\le C g$, and we write $f \asymp g$ or $f=\Theta(g)$ if $f\ll g$ and $g \gg f$.
We write $O_k$ to indicate that the implicit constant depends on $k$.
We write $o_{X\to +\infty}(g)$ to denote a quantity $f$ such that $f/g$ tends to zero as $X$ tends to infinity.

\subsection{Acknowledgements}

We thank Andrew Granville and the anonymous referee for many detailed comments that led us to significantly improve the results and presentation of our work.
We thank and Adam Harper for helpful discussions and useful comments and corrections on earlier versions.
We also thank Yuqiu Fu, Larry Guth, Kannan Soundararajan, Katharine Woo, and Liyang Yang for helpful discussions.
Finally, we thank Peter Sarnak for introducing us (the authors) to each other during the ``50 Years of Number Theory and Random Matrix Theory'' Conference at IAS and making the collaboration possible.

\section{Moments of random multiplicative functions in short intervals}
\label{SEC:high-moments-proof}

In this section, we prove Theorem~\ref{thm: high moments}.
For integers $k,n\ge 1$, let $\tau_k(n)$ denote the number of positive integer solutions $(d_1,\dots,d_k)$ to the equation $d_1\cdots d_k = n$.
It is known that (see \cite[Theorem~1.29 and Corollary~1.36]{Nor})
\begin{equation}
\label{INEQ:k-fold-divisor-bound}
\tau_k(n) \ll n^{O(\log{k}/\log\log{n})}
\;\textnormal{as $n\to +\infty$, provided $k = o_{n\to +\infty}(\log{n})$}.
\end{equation}

As we mentioned before, when $H\ge x$, Theorem~\ref{thm: high moments} is implied by \eqref{eqn: high}.
From now on, we focus on the case $H\le x$. 
We split the proof into two cases: small $H$ and large $H$.
For small $H$, we illustrate the general strategy and carelessly use divisor bounds;
for large $H$, we take advantage of bounds of Shiu \cite{Shiu80} and Henriot \cite{Henriot2012} on mean values and correlations of multiplicative functions over short intervals, together with a decomposition idea.

\subsection{Case 1: \texorpdfstring{$H\le x^{1-\ee k^{-1}}$}{H small}.}
Here we take $\ee$ to be a small absolute constant, e.g.~$\ee= \frac{1}{100}$. 

We begin with the following proposition.

\begin{proposition}
\label{PROP:easy-divisor-problem}
Let $k, y, H\ge 1$ be integers.
Suppose $y$ is large and $k\le \log\log{y}$.
Then $N_k(H;y)$, the number of integer tuples $(h_1, h_2, \dots, h_k) \in[-H, H]^{k}$ with $y|h_1h_2\cdots h_k$ and $h_1 h_2\cdots h_k \neq 0$,
is at most $(2H)^k \cdot O(\frac{H^{O(\frac{k\log k}{\log \log y})}}{y})$.  
\end{proposition}

\begin{proof}
The case $k=1$ is trivial; one has $N_1(H;y)\le 2H/y$. Suppose $k\ge 2$. 
Whenever $y|h_1h_2\cdots h_k\neq 0$, there exists a factorization $y=u_1u_2\cdots u_k$ where $u_i$ are positive integers such that $u_i|h_i\neq 0$ for all $1\le i \le k$.
(Explicitly, one can take $u_1 = \gcd (h_1, y)$ and $u_i = \gcd (h_i, \frac{y}{\gcd(y, h_1h_2\cdots h_{i-1})})$.)
It follows that $N_k(H; y) = 0$ if $y>H^k$, and
\begin{equation}
\label{INEQ:key-combo-ineq-from-easy-divisor-problem}
N_k(H; y) \le \sum_{u_1u_2\cdots u_k = y} N_1(H; u_1) N_1(H; u_2)\cdots N_1(H; u_k)
\le \tau_k(y)\cdot (2H)^k/y
\end{equation}
if $y\le H^k$.
By the divisor bound \eqref{INEQ:k-fold-divisor-bound}, Proposition~\ref{PROP:easy-divisor-problem} follows.
\end{proof}

\begin{corollary}\label{Cor: S_k}
Let $k, H, x\ge 1$ be integers.
Suppose $x$ is large and $k\le \log\log{x}$.
Then $S_k(x, H)$, the set of integer tuples $(h_1, h_2, \dots, h_k, y) \in [-H, H]^{k} \times (x, x+H]$ with $y|h_1h_2\cdots h_k$ and $h_1 h_2\cdots h_k \neq 0$,
has size at most $(2H)^{k} \cdot O(\frac{H^{1+O(\frac{k\log k}{\log \log{x}})}}{x})$. 
\end{corollary}

\begin{proof}
$\#S_k(x,H) = \sum_{x<y\le x+H} N_k(H;y)$.
But here $N_k(H;y)\ll (2H)^k\cdot \frac{H^{O(\frac{k\log k}{\log \log{x}})}}{x}$.
\end{proof}

The $2k$-th moment in Theorem~\ref{thm: high moments} is $H^{-k}$ times the point count \eqref{eqn: counting} for the Diophantine equation
\begin{equation}
\label{EQN:equal-products}
n_1n_2\cdots n_k = n_{k+1} n_{k+2} \cdots n_{2k}.
\end{equation}
There are $k! H^{k} (1 +  O(k^2/H)) = k! H^k + O_k(H^{k-1})$ trivial solutions.
(We call a solution to \eqref{EQN:equal-products} \emph{trivial} if
the tuple $(n_{k+1}, \dots n_{2k})$ equals a permutation of $(n_1, \dots n_k)$.)
The number of trivial solutions is clearly $\ge k! H(H-1)\cdots (H-k+1)$, and $\le k! H^k$.)
It remains to bound $N_k (x, H)$, the number of nontrivial solutions $(n_1,\dots,n_{2k}) \in (x, x+H]^{2k}$ to \eqref{EQN:equal-products}.

We will show that $N_k (x, H)\ll H^k\cdot (H/x)^{1/2}$.
To this end, let $N_k'(x, H)$ denote the number of nontrivial solutions in $(x, x+H]^{2k}$ with the further constraint that 
\begin{equation}\label{eqn: cond}
    n_{2k} \notin \{n_1, n_2, \dots, n_k\}.
\end{equation}
Then for any $k\ge 2$,
we have
\begin{equation}
\label{INEQ:recursion}
N_k(x, H) \le N'_{k}(x, H)  + k \cdot (H+1) \cdot N_{k-1}(x, H),
\end{equation}
since for each $(n_1,\dots,n_{2k}) \in (x, x+H]^{2k}$, either \eqref{eqn: cond} holds or there exists $i\in [k]$ satisfying $n_i=n_{2k} \in (x, x+H]$.

A key observation\footnote{After writing the paper, the authors learned that this observation has appeared before in the literature (see \cite[proof of Lemma~22]{Bourgain2014});
however, we take the idea further, both in \S\ref{SEC:high-moments-proof} and in \S\ref{SEC:proof-concentration-theorem}.}
is that for nontrivial solutions to \eqref{EQN:equal-products} with constraint \eqref{eqn: cond}, 
\[n_{2k} | (n_1-n_{2k}) (n_2- n_{2k}) \cdots (n_k - n_{2k}),\]
and if we write $h_i := n_i-n_{2k}$
then $h_i \in [-H, H]$ are nonzero.
Given $h_1, h_2, \dots, h_k, y$, let 
\[C_{h_1,\dots, h_k, y}: = \prod_{1\le i\le k} (h_i + y ). \]
Then $N'_{k}(x, H)$ is (upon changing variables from $n_1,\dots,n_k$ to $h_1,\dots,h_k$) at most
\begin{equation}
\label{EXPR:k-fold-divisor-sum-upper-bound-for-N'_k(x,H)}
    \sum_{\substack{(h_1, \dots, h_k, n_{2k})\in S_k(x, H) \\ h_i+n_{2k}>0}} \# \biggl\{(n_{k+1}, \dots, n_{2k-1}) \in (x, x+H]^{k-1} : \biggl(\,\prod_{i=1}^{k-1} n_{k+i}\biggr) \bigg| C_{h_1,\dots, h_k, n_{2k}}  \biggr\}.
\end{equation}
If $x$ is large and $k$ is fixed (or $k\le \log\log x$, say), then by the divisor bound \eqref{INEQ:k-fold-divisor-bound}, the quantity \eqref{EXPR:k-fold-divisor-sum-upper-bound-for-N'_k(x,H)} is at most 
\[ \ll (H+x)^{O(\frac{k\log k}{\log \log x} )} \cdot \# S_k (x, H)
\ll O(H)^{k} \cdot O(H\cdot x^{-1 +O(\frac{k\log k}{\log \log x}) }),\]
where in the last step we used Corollary~\ref{Cor: S_k}.

By \eqref{INEQ:recursion}, it follows that $x$ is large and $k$ is fixed (or $k\le \log\log x$, say), then
\begin{equation}\label{eqn: N_k}
 N_k(x,H) \le k\cdot \max_{1\le j\le k}(O(kH)^{k-j}\cdot N'_j(x,H))
\ll k\cdot O(kH)^{k} \cdot O(H\cdot x^{-1 +O(\frac{k\log k}{\log \log x}) }).    
\end{equation}
(Note that $N_1(x,H) = 0$.)
So in particular, $N_k (x, H)\ll H^k\cdot (H/x)^{1/2}$ for fixed $k$ (or for $x$ large and $k\le (\log\log{x})^{1/2-\delta}$, say), since $H\le x^{1-\ee k^{-1}}$.
This suffices for Theorem~\ref{thm: high moments}.

\begin{remark}
\label{RMK:k-uniformity}
The argument above in fact gives, in Case 1, a version of Theorem~\ref{thm: high moments} with an implied constant of $O(k! k^2)$, uniformly over $k\le (\log \log x)^{1/2-\delta}$, say.
However, in Case 2 below, our proof relies on a larger body of knowledge for which the $k$-dependence does not seem easy to work out; this is why we essentially keep $k$ fixed in Theorem~\ref{thm: high moments}.
\end{remark}

\subsection{Case 2: \texorpdfstring{$x^{1-2\ee k^{-1}} \le H \le x$}{H large}}
Again, one can assume $\ee=\frac{1}{100}$.
In this case, we employ the following tool due to Henriot \cite[Theorem~3]{Henriot2012}.
For the multiplicative functions $f$ in \eqref{INEQ:Henriot-simplification} (and in similar places below), we let $f(m) := 0$ if $m\le 0$.

\begin{definition}
Given a real $A_1\ge 1$ and a function $A_2=A_2(\epsilon)\ge 1$ (defined for reals $\epsilon>0$), let $\mathcal{M}(A_1,A_2)$ denote the set of non-negative multiplicative functions $f(n)$ such that $f(p^{\ell}) \le A_1^{\ell}$ (for all primes $p$ and integers $\ell\ge 1$) and $f(n)\le A_2 n^{\epsilon}$ (for all $n\ge 1$).
\end{definition}

\begin{lemma}
\label{LEM:Henriot}
Let $f_1, f_2\in \mathcal{M}(A_1,A_2)$ and $\beta \in (0,1)$.
Let $a,q\in \Z$ with $|a|,q\ge 1$ and $\gcd(a,q) = 1$.
If $x,y\ge 2$ are reals with $x^{\beta} \le y \le x$ and $x\ge \max(q,|a|)^\beta$, then
\begin{equation}
\label{INEQ:Henriot-simplification}
\sum_{x\le n\le x+y} f_1(n)f_2(qn+a)
\ll_{\beta,A_1,A_2} \Delta_D \cdot y
\cdot \sum_{n_1n_2\le x} \frac{f_1(n_1)f_2(n_2)}{n_1n_2},
\end{equation}
where $\Delta_D \le \prod_{p|a^2} (1+(2A_1+A_1^2)p^{-1})$.
Furthermore,
\begin{equation}
\label{INEQ:Henriot-Delta_D-bound}
\Delta_D \le \biggl(\frac{|a|}{\phi(|a|)}\biggr)^{\!2A_1+A_1^2}
\quad(\textnormal{where $\phi$ denotes Euler's totient function}).
\end{equation}
\end{lemma}

\begin{proof}
Everything but \eqref{INEQ:Henriot-Delta_D-bound} follows from \cite[Theorem~3]{Henriot2012} and the unraveling of definitions done in \cite[proof of Lemma~2.3(ii)]{KRT};
in the notation of \cite[Theorem~3]{Henriot2012}, we take
\begin{equation*}
(k, Q_1(n), Q_2(n), \alpha, \delta,
A, B, F(n_1,n_2))
= (2, n, qn+a, \tfrac{9}{10}\beta, \tfrac{9}{10}\beta,
A_1, A_2(\epsilon)^2, f_1(n_1)f_2(n_2)),
\end{equation*}
where $\epsilon = \frac{\alpha}{100(2+\delta^{-1})}$.\footnote{In fact, one could extract a more complicated version of \eqref{INEQ:Henriot-simplification} from \cite[Theorem~3]{Henriot2012}, which in some cases (e.g.~if $f_1=f_2=\tau_k$) would improve the right-hand side of \eqref{INEQ:Henriot-simplification} by roughly a factor of $\log{x}$.}
The inequality \eqref{INEQ:Henriot-Delta_D-bound} then follows from the fact that $1+rp^{-1} \le (1+p^{-1})^r \le (1-p^{-1})^{-r}$ for every prime $p$ and real $r\ge 1$.
\end{proof}



Also useful to us will be the following immediate consequence of Shiu \cite[Theorem~1]{Shiu80}.

\begin{lemma}
\label{LEM:Shiu-simple-case}
Let $f\in \mathcal{M}(A_1,A_2)$ and $\beta \in (0,1)$.
If $x,y\ge 2$ are reals with $x^{\beta} \le y \le x$, then
\begin{equation*}
\sum_{x\le n\le x+y} f(n)
\ll_{\beta,A_1,A_2} \frac{y}{\log x} \exp\biggl(\,\sum_{p\le x} \frac{f(p)}{p} \biggr).
\end{equation*}
\end{lemma}


We will apply the above results to $f=\tau_k$ over intervals of the form $[x, x+y]$ with $y\gg x^{1/2k}$, say.
Here $\tau_k\in \mathcal{M}(k,O_{k,\epsilon}(1))$, by \eqref{INEQ:k-fold-divisor-bound} and the fact that $\tau_k(p) = k$ and
\begin{equation}
\label{INEQ:tau_k-sub-multiplicative}
\tau_{k}(mn)\le \tau_{k}(m)\tau_k(n) \qquad\text{for arbitrary integers $m,n\ge 1$}.
\end{equation}
Also, recall, for integers $k\ge 1$ and reals $x\ge 2$, the standard bound
\begin{equation}
\label{INEQ:tau_k-first-moment}
\sum_{n\le x} \tau_k(n)
\ll_k \frac{x}{\log x} \exp\biggl (\,\sum_{p\le x} \frac{k}{p} \biggr)
\ll_k x (\log{x})^{k-1}
\end{equation}
(see e.g.~\cite[\S2.2]{KRT}) and the consequence
\begin{equation}
\label{INEQ:tau_k-convolution-first-moment}
\sum_{n_1n_2\le x} \tau_k(n_1)\tau_k(n_2)
= \sum_{n\le x} \tau_{2k}(n) \ll_k x (\log{x})^{2k-1}.
\end{equation}
(See \cite{Nor} for a version of \eqref{INEQ:tau_k-first-moment} with an explicit dependence on $k$.
For Lemmas~\ref{LEM:Henriot} and~\ref{LEM:Shiu-simple-case}, we are not aware of any explicit dependence on $\beta,A_1,A_2$ in the literature.)

\begin{lemma}
\label{LEM:average-linear-shift-correlation}
Let $V,U,q\ge 1$ be integers with $q\le U^{k-2}$, where $k\ge 2$.
Let $\rho\in \{-1,1\}$.
Then
\begin{equation*}
\sum_{\substack{u\in [U, 2U) \\ 1\le v\le V}}
\tau_k(u) \tau_k(\rho  v + uq)
\ll_k VU (1+\log{VU})^{3k}.
\end{equation*}
\end{lemma}

\begin{proof}
First suppose $V\ge U$. 
If $u\in [U,2U)$, then $I := \{\rho v+uq: 1\le v\le V\}$ is an interval of length $V\ge \max(V, U)$ contained in $[-V, V + 2U^{k-1}]$, so by Lemma~\ref{LEM:Shiu-simple-case} and \eqref{INEQ:tau_k-first-moment}, we obtain the bound
\begin{equation*}
\sum_{1\le v\le V} \tau_k(\rho v+uq) \ll_k V (1+\log{V})^{k-1}.
\end{equation*}
(We consider the cases $0\in I$ and $0\notin I$ separately.
The former case follows directly from \eqref{INEQ:tau_k-first-moment};
the latter case requires Lemma~\ref{LEM:Shiu-simple-case}.)
Then sum over $u$ using \eqref{INEQ:tau_k-first-moment}.
Since $(1+\log{V})^{k-1}(1+\log{U})^{k-1} \le (1+\log{VU})^{2k-2}$, Lemma~\ref{LEM:average-linear-shift-correlation} follows.

Now suppose $V\le U$.
By casework on $d := \gcd(v,q)\le q$, we have
\begin{equation*}
\sum_{\substack{u\in [U, 2U) \\ 1\le v\le V}}
\tau_k(u) \tau_k(\rho  v + uq)
\le \sum_{d|q} \tau_k(d)
\sum_{\substack{u\in [U, 2U) \\ 1\le a\le V/d \\ \gcd(a, q/d) = 1}} \tau_k(u)\tau_k(\rho a + uq/d).
\end{equation*}
Since $d|q$ and $1\le a\le V/d$, we have $U \ge \max(a, q^{1/k})$. Now for any fixed $1\le a\le V/d$, 
\begin{equation*}
\sum_{u\in [U, 2U)} \tau_k(u)\tau_k(\rho a + uq/d)
\ll_{k} \biggl(\frac{a}{\phi(a)}\biggr)^{\!2k+k^2} \cdot U \cdot (1+\log{U})^{2k}
\end{equation*}
by Lemma~\ref{LEM:Henriot} and \eqref{INEQ:tau_k-convolution-first-moment}, provided $\gcd(a, q/d) = 1$.
Upon summing over $1\le a\le V/d$ using \cite[p.~61, (2.32)]{montgomery2007multiplicative}, it follows that
\begin{equation*}
\sum_{\substack{u\in [U, 2U) \\ 1\le v\le V}}
\tau_k(u) \tau_k(\rho  v + uq)
\ll_k \sum_{d|q} \tau_k(d)\cdot \frac{V}{d}\cdot U \cdot (1+\log{U})^{2k}.
\end{equation*}
Since $\sum_{d\le q} \frac{\tau_k(d)}{d} \ll_k (1+\log{q})^k$ (by \eqref{INEQ:tau_k-first-moment}) and
 $q\le U^{k-2}$, Lemma~\ref{LEM:average-linear-shift-correlation} follows.
\end{proof}

\begin{lemma}
\label{LEM:average-multi-linear-shift-correlation}
Let $V_1,U_2,\dots,U_k\ge 1$ be integers, where $k\ge 2$.
Let $\ee_1\in \{-1,1\}$.
Then
\begin{equation*}
\sum_{\substack{v_1,u_2,\dots,u_k\ge 1 \\ u_i\in [U_i, 2U_i) \\ v_1\le V_1}}
\tau_k(u_2)\cdots \tau_k(u_k)  \tau_k(\ee_1  v_1 + u_2\cdots u_k)
\ll_k L_k(V_1U_2\cdots U_k),
\end{equation*}
where $L_k(r) := r\cdot (1+\log{r})^{3k + (k-2)(k-1)} = r\cdot (1+\log{r})^{k^2+2}$ for $r\ge 1$.
\end{lemma}

\begin{proof}
We may assume $U_2\ge \cdots \ge U_k$.
Let $q := u_3\cdots u_k\le U_2^{k-2}$ and apply Lemma~\ref{LEM:average-linear-shift-correlation} (with $(V,U)=(V_1,U_2)$) to sum over $u_2,v_1$.
Then sum over the $k-2$ variables $u_3,\dots,u_k$ using \eqref{INEQ:tau_k-first-moment}.
\end{proof}

With the lemmas above in hand,
we now build on the strategy from Case 1 to attack Case 2.
As before, we let $N'_k(x, H)$ denote the number of nontrivial solutions $(n_1, \dots, n_k, n_{k+1}, \dots, n_{2k}) \in (x, x+H]^{2k}$ to \eqref{EQN:equal-products} with constraint~\eqref{eqn: cond}.
Again, for such solutions we write $h_i=n_i-n_{2k}\in [-H,H]\setminus \{0\}$, and there exist positive integers $u_i$ ($1\le i\le k$) such that $u_i|h_i$ with $u_1u_2\cdots u_k = n_{2k}\in (x,x+H]$;
so $u_i\le H$, and there exist signs $\ee_i\in \{-1, 1\}$ and positive integers $v_i\le H/U_i$ with $h_i = \ee_i u_i v_i$, whence
\[C_{h_1,\dots, h_k, n_{2k}}: = \prod_{i=1}^{k} (h_i + n_{2k})
= \prod_{1\le i \le k} (\ee_i u_iv_i + u_1u_2\cdots u_k). \]

As before, the quantity $N'_k(x,H)$ is at most \eqref{EXPR:k-fold-divisor-sum-upper-bound-for-N'_k(x,H)}.
Upon splitting the range $[H]$ for each $u_i$ into $\le 1+\log_2{H}\ll 1+\log{x}$ dyadic intervals, we conclude that
\begin{equation}
\label{INEQ:dyadic-bound-N'_k(x,H)}
    N'_k(x,H) \le \sum_{\ee_i,U_i}
    \sum_{\substack{u_i\in [U_i, 2U_i) \\ v_i\le H/U_i \\ x < n_{2k} \le x+H \\  h_i+n_{2k}>0}} \tau_k (C_{h_1,\dots, h_k, n_{2k}})
    \le 2^k\cdot O(1+\log{x})^k \cdot \mathcal{S}(x,H),
\end{equation}
where we let $n_{2k} := u_1u_2\cdots u_k$ and $h_i := \ee_i u_i v_i$ in the sum over $u_i,v_i$ (for notational brevity), and where $\mathcal{S}(x,H)$ denotes the maximum of the quantity
\begin{equation*}
    S(\vec{\ee},\vec{U}) := \sum_{\substack{u_i\in [U_i, 2U_i) \\ v_i\le H/U_i \\ x < n_{2k} \le x+H \\  h_i+n_{2k}>0}} \tau_k (C_{h_1,\dots, h_k, n_{2k}})
    = \sum_{\substack{u_i\in [U_i, 2U_i) \\ v_i\le H/U_i \\ x < n_{2k} \le x+H \\  h_i+n_{2k}>0}}
    \tau_k\biggl(\,\prod_{1\le i \le k} (\ee_i u_iv_i + u_1u_2\cdots u_k)\biggr)
\end{equation*}
over all tuples $\vec{\ee}=( \ee_1,\dots,\ee_k)\in \{-1,1\}^{k}$ and $\vec{U} = (U_1,\dots,U_k)$ where each $U_i \in [H]\cap \{1,2,4,8,\dots\}$ with $2^{-k}x < U_1\cdots U_k\le x+H$.
Now, for the rest of \S\ref{SEC:high-moments-proof}, fix a choice of $\ee_1,\dots,\ee_k,U_1,\dots,U_k$ with
\begin{equation*}
    \mathcal{S}(x,H) = S(\vec{\ee},\vec{U}).
\end{equation*}
By symmetry, we may assume that $U_1\ge U_2\ge \cdots \ge U_k$.

We now bound $S(\vec{\ee},\vec{U})$, assuming $k\ge 2$.
(For $k=1$, we can directly note that $N'_1(x,H) = 0$.)
A key observation is that since $U_1U_2\cdots U_k\le x+H\le 2x$ and $U_1\ge U_2\ge \cdots \ge U_k\ge 1$, we have (since $H\ge x^{1-2\ee}$ and $k\ge 2$)
\[\frac{H}{U_k}\ge \frac{H}{U_{k-1}} \ge \cdots \ge \frac{H}{U_2} \ge \frac{H}{(U_1U_2)^{1/2}} \ge \frac{x^{1-2\ee}}{(2x)^{1/2}}\gg x^{1/3}.  \]
By the sub-multiplicativity property \eqref{INEQ:tau_k-sub-multiplicative}, we have that $S(\vec{\ee},\vec{U})$ is at most
\begin{equation}
\label{EXPR:upper-bound-for-S(ee,U)}
\sum_{\substack{u_i\in [U_i, 2U_i)\\ x < u_1u_2\cdots u_k \le x+H}}
\sum_{\substack{v_i\le H/u_i}}
\tau_k(u_1) \tau_k(u_2)\cdots \tau_k(u_k) \prod_{1\le i \le k} \tau_k(\ee_i  v_i + u_1u_2\cdots u_{-i} \cdots u_{k}),
\end{equation}
where $u_{-i}$ means that the factor $u_i$ is not included. 
But for each $i\ge 2$ and $u_i\in [U_i,2U_i)$, Lemma~\ref{LEM:Shiu-simple-case} and \eqref{INEQ:tau_k-first-moment} imply (since $u_1u_2\cdots u_{-i} \cdots u_k \le u_1\cdots u_k\ll x$ and $H/u_i\gg x^{1/3}$)
\begin{equation}
\label{EQN:average-over-sizable-v_i}
\sum_{v_i\le H/u_i} \tau_k(\ee_i v_i + u_1u_2\cdots u_{-i} \cdots u_k) \ll_k (H/U_i)\cdot (1+\log{x})^{k-1};
\end{equation}
cf.~the use of Lemma~\ref{LEM:Shiu-simple-case} and \eqref{INEQ:tau_k-first-moment} in the proof of Lemma~\ref{LEM:average-linear-shift-correlation}.
By \eqref{EQN:average-over-sizable-v_i} (multiplied over $2\le i\le k$) and Lemma~\ref{LEM:average-multi-linear-shift-correlation} (with $V_1=H/U_1$), we conclude that the quantity \eqref{EXPR:upper-bound-for-S(ee,U)} (and thus $S(\vec{\ee},\vec{U})$) is at most
\[ \ll_k \frac{H^{k-1} (1+\log{x})^{(k-1)^2}}{U_2\cdots U_k}
\cdot L_k((H/U_1)\cdot U_2\cdots U_k)
\cdot \max_{\substack{u_2,\dots,u_k\ge 1 \\ u_i\in [U_i, 2U_i)}}
\sum_{\substack{u_1\in [U_1, 2U_1) \\ x < u_1u_2\cdots u_k \le x+H}} \tau_k(u_1). \]
For the innermost sum, first note that $(U_2\cdots U_k)^{1/(k-1)}\le (U_1\cdots U_k)^{1/k} \le (2x)^{1/k}$ which implies that
\begin{equation*}
    H/(u_2\cdots u_k)\gg_k H/(U_2\cdots U_k)\gg_k x^{1-2\ee k^{-1}}/x^{(k-1)/k} \ge x^{1/2k}
\end{equation*}
(since $H\ge x^{1-2\ee k^{-1}}$);
then by Lemma~\ref{LEM:Shiu-simple-case} and \eqref{INEQ:tau_k-first-moment}, we have (for any given $u_2,\dots,u_k$)
\begin{equation*}
    \sum_{\substack{u_1\ge 1 \\ x < u_1u_2\cdots u_k \le x+H}} \tau_k(u_1) \ll_k \frac{H}{U_2\cdots U_k} \cdot (1+\log{x})^{k-1}.
\end{equation*}
It follows that $S(\vec{\ee},\vec{U})$ is at most
\begin{equation*}
   \ll_k \frac{H^{k-1} (1+\log{x})^{(k-1)^2}}{U_2\cdots U_k}
    \cdot \frac{H}{U_1}\cdot U_2\cdots U_k (1+\log{x})^{k^2+2}
    \cdot \frac{H}{U_2\cdots U_k} \cdot (1+\log{x})^{k-1},
\end{equation*}
which simplifies to $O_k(1)\cdot H^k\cdot (H/x)\cdot (1+\log{x})^{2k^2-k+2}$.

Plugging the above estimate into \eqref{INEQ:dyadic-bound-N'_k(x,H)}, we have (assuming $k\ge 2$)
\begin{equation}\label{eqn: N'_k}
  N'_k(x, H) \ll_k O(1+\log x)^{k}\cdot \mathcal{S}(x,H)
  \ll_k H^k \cdot (H/x) \cdot (1+\log{x})^{2k^2+2},
\end{equation}
in the given range of $H$.
Then by using the first part of \eqref{eqn: N_k} (and noting that $N_1(x,H) = N'_1(x,H) = 0$) as before, we have (for arbitrary $k\ge 1$)
\[N_k(x,H) \le k\cdot \max_{1\le j\le k}(O(kH)^{k-j}\cdot N'_j(x,H))\ll_k H^k \cdot (H/x) \cdot (1+\log{x})^{2k^2+2}, \]
which suffices for Theorem~\ref{thm: high moments}.

\section{Proof of Theorem~\ref{thm: Steinhaus}}
\label{SEC:proof-concentration-theorem}

In this section, we prove Theorem~\ref{thm: Steinhaus}.
Let $\rad_k$ be the multiplicative function 
\[
    \rad_k(n) = \min_{n_1\cdots n_k = n} [n_1,\dots,n_k],
\]
where $[n_1,\dots,n_k]$ denotes the least common multiple of $n_1,\dots,n_k$.
In particular, for prime powers $p^\ell$ we have
\begin{equation}
\label{EQN:compute-rad_k(p^ell)-prime-powers}
\rad_k(p^\ell) = p^{\ceil{\ell/k}}.
\end{equation}

Recall that we use $\tau_k(n)$ to denote the $k$-folder divisor function as defined in \eqref{EXPR:k-fold-divisor-sum-upper-bound-for-N'_k(x,H)}.
We begin with the following sequence of lemmas.

\begin{lemma}
\label{LEM:easy-polynomial-problem}
Let $k, y, X, H\ge 1$ be integers.
Then $M_k(X,H;y) := \{(x, t_1, t_2, \dots, t_k) \in [X]\times [H]^{k}: y|(x+t_1)(x+t_2)\cdots (x+t_k)\}$
has size at most $H^k \tau_k(y)\cdot (1 + X/\rad_k(y))$.  
\end{lemma}

\begin{proof}
Suppose that $y | (x + t_1)\dots (x + t_k)$.
Then there exist integers $y_1,\dots,y_k\ge 1$ with
$y_1\cdots y_k = y$ and $y_i | x + t_i$ ($1\le i\le k$). 

For any given choice of $y_1,\dots,y_k,t_1,\dots,t_k$, the conditions $y_i | x + t_i$, when satisfiable,
impose on $x$ a congruence condition modulo $[y_1,\dots,y_k]$. It follows that for any given $t_1,\dots,t_k$,
the number of values of $x\in [X]$ with $(x,t_1,\dots,t_k)\in M_k(X,H;y)$ is at most
\[
    \sum_{y_1\cdots y_k = y} (1 + X/[y_1,\dots,y_k])
    \le \tau_k(y)\cdot (1 + X/\rad_k(y)).
\]
Lemma~\ref{LEM:easy-polynomial-problem} follows upon summing over 
$t_1,\dots,t_k\in [H]$.
\end{proof}
\begin{remark}
For a typical value of $y\le X$, Lemma~\ref{LEM:easy-polynomial-problem} saves a factor of roughly $y$ over the trivial bound $H^k X$, even if $H\le X^{1-\delta}$, say. Lemma~\ref{LEM:easy-polynomial-problem} is close to optimal on average over $y\le X$, as one can prove by considering prime values of $y$, for instance.
In some regimes, one can do better by other arguments:
one can first fix a choice of $y_i$ (then select $x$  and choose $t_i\equiv -x \mod y_i$) to get
\[|M_k(X,H;y)|
\le \sum_{y_1\cdots y_k = y} X\prod_{i}(1 + H/y_i)
\le \tau_k(y) X \max_{y_1\cdots y_k=y} \prod_{i}(1 + H/y_i), \]
which beats Lemma~\ref{LEM:easy-polynomial-problem} when $H\ge y$ and $y/\rad_k(y)$ is large, but not in general.
\end{remark}

\begin{lemma}\label{Lem: yth}
Let $k, y, X, H\ge 1$ be integers.
Then $B_k(X,H;y)$, the set of integer tuples $(x, t_1, \dots, t_k, h_1, \dots, h_k) \in [X]\times [H]^k \times [-H, H]^k$ with $y|(x+t_1)(x+t_2)\cdots (x+t_k)h_1h_2\cdots h_k$ and $h_1h_2\cdots h_k\neq 0$,
has size at most $O(H)^{2k} \cdot \tau_2(y)\tau_k(y)^2\cdot O(1 + X/\rad_k(y))$.
\end{lemma}

\begin{proof}
We write $y=uv$ with $u|(x+t_1)(x+t_2)\cdots (x+t_k)$ and $v|h_1h_2\cdots h_k$ (where $u,v\ge 1$).
The number of choices of $(u, v)$ is $\le \tau_2(y)$.
Using the notation in Lemma~\ref{LEM:easy-polynomial-problem} and Proposition~\ref{PROP:easy-divisor-problem}, we then find that
\[ |B_k(X,H;y)|
\le \sum_{uv=y} |M_k(X, H; u)| \cdot N_k(H; v)
\le \tau_2(y) \max_{uv=y} |M_k(X, H; u)| \cdot N_k(H; v). \]
Now for any fixed $u, v$, we apply Lemma~\ref{LEM:easy-polynomial-problem} to bound $|M_k(X, H; u)|$ and \eqref{INEQ:key-combo-ineq-from-easy-divisor-problem} to bound $N_k(H; v)$, getting
\[|M_k(X, H; u)|\le H^{k} \tau_k(u)\cdot (1+X/\rad_k(u))
\quad \text{and} \quad N_k(H; v) \le (2H)^{k} \tau_k(v)/v,  \]
respectively.
This leads to the total bound
\[|B_k(X,H;y)|\ll \tau_2 (y) H^{2k} \tau_k(y)^{2} \cdot \l(1+ \frac{X}{v\rad_k(u)}\r). \]
For any $uv=y$, we have
\[v\rad_k(u) \ge \rad_k(y),  \]
by the multiplicativity of $\rad_k$, the formula \eqref{EQN:compute-rad_k(p^ell)-prime-powers}, and the inequality $p^{\ell_2} p^{\ceil{\ell_1/k}} \ge p^{\ceil{(\ell_1+\ell_2)/k}}$ (valid for all primes $p$ and integers $\ell_1,\ell_2\ge 0$).
Thus we complete the proof.
\end{proof}

If we allowed $h_1h_2\cdots h_k=0$, we would have $X\cdot O(H)^{2k-1}$ tuples in $B_k(X,H;y)$.
Lemma~\ref{Lem: yth} gives a relative saving of roughly $y/H$ on average over $y\ll X$;
this follows from (the proof of) Lemma~\ref{Lem: average yx} below, whose proof requires the following lemma.

\begin{lemma}
\label{LEM:local-Euler-factor-computation}
Let $K,k\ge 2$ be integers.
For integers $i\ge 1$, let
\[  c_i := \sum_{(i - 1)k < j \le ik}\binom{j + K - 1}{K - 1}.
\]
Then $c_i\le K^K (ik)^K$.
Furthermore, for all primes $p$ and reals $s>1$, we have
\begin{equation*}
\sum_{j\ge 1} \tau_{K}(p^j)\frac{p^j}{\rad_k(p^j)} p^{-js}
\le 1 + \frac{c_1}{p^s} + \frac{c_2}{p^{2s}} + \cdots.
\end{equation*}
\end{lemma}

\begin{proof}
The first part is clear, since $c_i\le \sum_{0\le j\le ik} \binom{j+K-1}{K-1} = \binom{ik + K}{K} \le (K+ik)^K \le K^K (ik)^K$ (since $K,k\ge 2$).
The second part follows from the inequality
\begin{equation*}
\sum_{(i - 1)k < j \le ik} \frac{\tau_K(p^j) p^j}{\rad_k(p^j) p^{js}}
= \sum_{(i - 1)k < j \le ik} \frac{\binom{j+K-1}{K-1}}{p^{\ceil{j/k}} p^{j(s-1)}}
\le \sum_{(i - 1)k < j \le ik} \frac{\binom{j+K-1}{K-1}}{p^{i} p^{i(s-1)}}
= \frac{c_i}{p^{is}},
\end{equation*}
which holds because we have $\ceil{j/k} = i$ and $j\ge i$ whenever $(i - 1)k < j \le ik$.
\end{proof}

It turns out that to prove the key Lemma~\ref{LEM:avg-cross-count} (below) for Theorem~\ref{thm: Steinhaus}, we need a bound of the form \eqref{INEQ:tau-rad_k-hybrid-log-moment-bound}. 

\begin{lemma}\label{Lem: average yx}
Let $k,X,H\ge 1$ be integers with $X$ large and $H\le X$. There exists a positive integer $\mathcal{C}_k = O(k^{O(k^{O(k)})})$ (depending only on $k$) such that the following holds:
\begin{equation}
\label{INEQ:tau-rad_k-hybrid-log-moment-bound}
\E_{x\in [X]} \sum_{y\in (x,x+H]} \tau_{2k}(y)^{2k}
\cdot \tau_2(y)\tau_k(y)^2\cdot (1 + X/\rad_k(y))
\ll_k H (\log{X})^{\mathcal{C}_k}.
\end{equation}
\end{lemma}

\begin{proof}
The case $k = 1$ is clear by \eqref{INEQ:tau_k-first-moment} (since $\rad_1(y) = y$), so suppose $k\ge 2$ for the remainder of this proof.
Let $K := (2k)^{2k}\cdot 2 k^2\le k^{4k+3}$.
Then $\tau_{2k}(y)^{2k} \tau_2(y)\tau_{k}(y)^2 \le \tau_{K}(y)$, since for all integers $j_1,j_2\ge 1$ we have $\tau_{j_1}(y)\tau_{j_2}(y)\le \tau_{j_1j_2}(y)$ by \cite[(3.2)]{BNR}.
By Rankin's trick, the left-hand side of \eqref{INEQ:tau-rad_k-hybrid-log-moment-bound} is therefore at most $H$ times
\begin{equation*}
\sum_{y\le x+H} \tau_{K}(y) \cdot (X^{-1} + \rad_k(y)^{-1})
\ll_K (\log{X})^{K-1} + \sum_{n\ge 1} \tau_{K}(n)\frac{n}{\rad_k(n)} n^{-1 - 1/\log X}.
\end{equation*}
By Lemma~\ref{LEM:local-Euler-factor-computation} and the multiplicativity of $\tau_{K}$ and $\rad_{k}$, we find that for $s > 1$, we have
\begin{equation}\label{eqn: rad_euler_product}
    \sum_{n\ge 1} \tau_{K}(n)\frac{n}{\rad_k(n)} n^{-s}
    \le \prod_{p\ge 2} \biggl(1 + \frac{c_1}{p^s} + \frac{c_2}{p^{2s}} + \cdots\biggr),
\end{equation}
where $c_i\le K^K (ik)^K \le K^{2K} (2K)^K 2^{i/2}$ (since $k\le K$ and $i^K/2^{i/2} \le (2K/\log{2})^K/e^K$, and $e\log{2}\ge 1$).
But then
\[
    \frac{c_2}{p^2} + \frac{c_3}{p^3} + \dots\ll\frac{K^{4K}}{p^2}.
\]
Therefore, the right-hand side of \eqref{eqn: rad_euler_product} is at most
\[
    \prod_{p\ge 2} \biggl(1 + \frac{1}{p^s}\biggr)^{\!c_1}
    \prod_{p\ge 2}\biggl(1 + \frac{1}{p^2}\biggr)^{\!O(K^{4K})}.
\]
After plugging in $s = 1+1/\log{X}$ and the bound $c_1\le K^{2K}$, Lemma~\ref{Lem: average yx} follows.
\end{proof}

We also need a simple but finicky combinatorial estimate.

\begin{lemma}
\label{LEM:count-equal-sets-without-multiplicities}
Let $k,x,H\ge 1$ be integers.
Let $\mathcal{A}_{1,2}(x,H)$ be the number of tuples $(a_1,\dots,a_{2k})\in (x,x+H]^{2k}$ satisfying both
\begin{enumerate}
    \item $\{a_1,\dots,a_k\} = \{a_{k+1},\dots,a_{2k}\}$ (in the usual sense, without multiplicities), and
    
    \item $a_1\cdots a_k = a_{k+1}\cdots a_{2k}$.
\end{enumerate}
Let $\mathcal{A}_1(x,H)$ be the number of tuples $(a_1,\dots,a_{2k})\in (x,x+H]^{2k}$ satisfying (1) (but not necessarily (2)).
Then $\mathcal{A}_{1,2}(x,H) \ge k!H^{k} - O_k(H^{k-1})$ and $\mathcal{A}_1(x,H)\le k!H^{k} + O_k(H^{k-1})$.
\end{lemma}

\begin{proof}
Call a tuple $(a_1,\dots,a_{2k})\in (x,x+H]^{2k}$ good if it satisfies (1).
Let $\mathcal{A}_1^\star$ be the number of good tuples where $a_1,\dots,a_k$ are pairwise distinct.
Let $\mathcal{A}_1^\dagger$ be the number of remaining good tuples, namely good tuples where $\prod_{1\le i<j\le k} (a_i-a_j) = 0$.
Then $\mathcal{A}_1 \le \mathcal{A}_1^\star + \mathcal{A}_1^\dagger$.

Clearly $\mathcal{A}_1^\star = k! H(H-1)\cdots (H-k+1)$ (since when the $a_i$ are all different for $1\le i\le k$, condition~(1) implies that $(a_{k+1},\dots,a_{2k})$ is a permutation of $(a_1,\dots,a_k)$;
and conversely, when $(a_{k+1},\dots,a_{2k})$ is a permutation of $(a_1,\dots,a_k)$, both (1) and (2) hold).
Furthermore, $\mathcal{A}_{1,2} \ge \mathcal{A}_1^\star$.

On the other hand, $\mathcal{A}_1^\dagger \le \binom{H}{k-1} \cdot (k-1)^{2k}$ (since if $\prod_{1\le i<j\le k} (a_i-a_j) = 0$, then $\{a_1,\dots,a_k\}$ must lie in some $(k-1)$-element subset $S\subseteq (x,x+H]$,
and then condition~(1) implies that each of $a_1,\dots,a_{2k}$ is an element of $S$).

We now know $\mathcal{A}_1^\star = k!H^k + O_k(H^{k-1})$
and $\mathcal{A}_1^\dagger\ll_k H^{k-1}$.
So $\mathcal{A}_{1,2} \ge \mathcal{A}_1^\star \ge k!H^k - O_k(H^{k-1})$,
and $\mathcal{A}_1 \le \mathcal{A}_1^\star + \mathcal{A}_1^\dagger \le k!H^k + O_k(H^{k-1})$.
\end{proof}

Given integers $x_1,x_2,H\ge 1$, let $I_j = (x_j, x_j +H]$ for $j\in \{1,2\}$.
We are now ready to estimate the size of the set
\begin{equation}\label{SET: cross solutions}
    \begin{split}
\{ (n_1, n_2,\dots, n_{2k}; m_1, m_2,\dots, m_{2k}) &\in I_1^{2k}\times I_2^{2k}:\\ & n_1\cdots n_k m_1\cdots m_k = n_{k+1} \cdots n_{2k} m_{k+1}\cdots m_{2k} \}.
\end{split}
\end{equation}

\begin{lemma}
\label{LEM:avg-cross-count}
Fix an integer $k\ge 1$;
let $\mathcal{C}_k$ be as in Lemma~\ref{Lem: average yx}.
Let $X,H$ be large integers with $H:= H(X) \to +\infty$ as $X\to +\infty$. 
Suppose $H\ll X(\log X)^{-2\mathcal{C}_k}$.
Then in expectation over $x_1,x_2\in [X]$, the size of the set \eqref{SET: cross solutions} is $k!^2H^{2k} + o_{X\to +\infty}(H^{2k})$.
\end{lemma}

\begin{proof}
We roughly follow \S\ref{SEC:high-moments-proof}'s proof of Theorem~\ref{thm: high moments};
however, the present situation is more complicated in some aspects, which we address using some new symmetry tricks.

First, let $T^\star_k(I_1,I_2)$ be the subset of \eqref{SET: cross solutions} satisfying the following conditions:
\begin{enumerate}
    \item If $u\in \{m_{k+1},\dots,m_{2k}\}$,
    then $u\in \{m_1,\dots,m_k\}$.
    
    \item If $u\in \{m_1,\dots,m_k\}$,
    then $u\in \{m_{k+1},\dots,m_{2k}\}$.
    
    \item If $u\in \{n_{k+1},\dots,n_{2k}\}$,
    then $u\in \{n_1,\dots,n_k\}$.
    
    \item If $u\in \{n_1,\dots,n_k\}$,
    then $u\in \{n_{k+1},\dots,n_{2k}\}$.
\end{enumerate}
In the notation of Lemma~\ref{LEM:count-equal-sets-without-multiplicities}, applied with $a=m$ and $a=n$ (separately), we have $\#T^\star_k(I_1,I_2) \ge \mathcal{A}_{1,2}(x_1,H) \mathcal{A}_{1,2}(x_2,H)$ and $\#T^\star_k(I_1,I_2) \le \mathcal{A}_1(x_1,H) \mathcal{A}_1(x_2,H)$, so
\begin{equation}
\label{INEQ:trivial-type-solution-estimate}
\#T^\star_k(I_1,I_2) = (k!H^{k} + O_k(H^{k-1}))^2 = k!^2H^{2k} + O_k(H^{2k-1}).
\end{equation}

In general, given an element $\mathfrak{n}\in I_1^{2k}\times I_2^{2k}$ of \eqref{SET: cross solutions}, let $\mathcal{U}$ be the set of integers $u$ that violate at least one of the conditions (1)--(4) above.
Then $\mathfrak{n}\in T^\star_k(I_1,I_2)$ if and only if $\mathcal{U} = \emptyset$.
This simple observation will help us estimate the size of \eqref{SET: cross solutions}.

Let $N^\star_k(I_1,I_2)$ be the subset of \eqref{SET: cross solutions} satisfying the following conditions:
\begin{enumerate}
    \item $n_{2k} \notin \{n_1, \dots, n_k\}$.
    (This implies, but is not equivalent to, $n_{2k}\in \mathcal{U}$.)
    
    
    
    \item If $u\in \mathcal{U}$, then $\tau_{2k}(u)\le \tau_{2k}(n_{2k})$.
\end{enumerate}
Then \eqref{SET: cross solutions} has size at least $\#T^\star_k(I_1,I_2)$ and we claim that \eqref{SET: cross solutions} has size at most
\[ \le \#T^\star_k(I_1,I_2) + 2k\cdot \#N^\star_k(I_1,I_2) + 2k\cdot \#N^\star_k(I_2,I_1).\]
First note that for each element $\mathfrak{n}$ of \eqref{SET: cross solutions} lying outside of $T^\star_k(I_1,I_2)$, there exist $v\in \mathcal{U}$ and $(a, b, c) \in \{m,n\} \times \{0,k\} \times [k]$, with $\tau_{2k}(v) = \max_{u\in \mathcal{U}} \tau_{2k}(u)$, such that $a_{b+c}=v$ and $a_{b+c}\notin \{a_{(k-b)+i}: i\in [k]\}$;
the existence of $v$ with $\tau_{2k}(v) = \max_{u\in \mathcal{U}} \tau_{2k}(u)$ follows from the fact that $\mathcal{U}\ne \emptyset$,
and the existence of  $(a,b,c)$ then follows from the definition of $\mathcal{U}$.
The claim then follows from the definitions of $N^\star_k(I_1,I_2)$ and $N^\star_k(I_2,I_1)$, upon summing over all possibilities for $a$, $b$, $c$.

It follows that in  expectation over $x_1,x_2\in [X]$, the size of \eqref{SET: cross solutions} is
\begin{equation}
\label{EQN:recursive-estimate}
\E_{x_1,x_2} \#T^\star_k(I_1,I_2) + O(2k\cdot \E_{x_1,x_2} \#N^\star_k(I_1,I_2)).
\end{equation}

The projection $I_1^{2k}\times I_2^{2k}\ni (n_1,\dots, n_{2k}; m_1,\dots, m_{2k})\mapsto (n_1,\dots,n_k;m_1,\dots,m_k;n_{2k}) \in I_1^k \times I_2^k \times I_1$,
i.e.~``forgetting'' $n_{k+1},\dots,n_{2k-1},m_{k+1},\dots,m_{2k}$, defines a map $\pi$ from $N^\star_k(I_1,I_2)$ to the set
\begin{equation*}
\begin{split}
D^\star_k(I_1,I_2)
:= \{(n_1,\dots,n_k;m_1,\dots,m_k;n_{2k}) \in I_1^k \times I_2^k \times I_1:{}
&n_{2k}|n_1\cdots n_km_1\cdots m_k, \\
&n_{2k} \notin \{n_1, \dots, n_k\}\}.
\end{split}
\end{equation*}
We now bound the fibers of $\pi$.
Suppose $(n_1,\dots,n_{2k};m_1,\dots,m_{2k})\in N^\star_k(I_1,I_2)$.
Let $S_1 := \{i\in \{k+1,\dots,2k\}: n_i\notin \mathcal{U}\}$
and $S_2 := \{j\in \{k+1,\dots,2k\}: m_j\notin \mathcal{U}\}$,
and let
\[
z := \prod_{i\in \{k+1,\dots,2k\}\setminus S_1} n_i
\prod_{j\in \{k+1,\dots,2k\}\setminus S_2} m_j
= \frac{n_1\cdots n_km_1\cdots m_k}{\prod_{i\in S_1} n_i \prod_{j\in S_2} m_j}.
\]
Then the following hold:
\begin{itemize}
    \item $n_i\in \{n_1,\dots,n_k\}$ for all $i\in S_1$,
    and $m_j\in \{m_1,\dots,m_k\}$ for all $j\in S_2$;
    \item $z$ depends only on $n_1,\dots,n_k,m_1,\dots,m_k,(n_i)_{i\in S_1},(m_j)_{j\in S_2}$;
    and
    \item $\tau_{2k-|S_1|-|S_2|}(z)
    \le \tau_{2k}(z)\le \tau_{2k}(n_{2k})^{2k-|S_1|-|S_2|}$.
    (The upper bound on $\tau_{2k}(z)$ arises as follows:
    since $z$ is the product of $2k-|S_1|-|S_2|$ elements $u_l$ of $\mathcal{U}$, we have an upper bound $\le \prod_{1\le l \le 2k- |S_1|-|S_2|} \tau_{2k}(u_l)$, which is $ \le \prod_{1\le l \le 2k- |S_1|-|S_2|} \tau_{2k}(n_{2k}) $.)
\end{itemize}
Therefore, the fiber of $\pi$ over $(n_1,\dots,n_k;m_1,\dots,m_k;n_{2k})\in D^\star_k(I_1,I_2)$ has size at most
\begin{equation}
\label{EXPR:combinatorial-symmetry-trick-factor}
\sum_{S_1,S_2\subseteq \{k+1,\dots,2k\}} k^{|S_1|} \cdot k^{|S_2|} \cdot \tau_{2k}(n_{2k})^{2k-|S_1|-|S_2|}
= \sum_{0\le l\le 2k} \binom{2k}{l} k^{l} \tau_{2k}(n_{2k})^{2k-l},
\end{equation}
where each $S_t$ ($1\le t\le 2$) runs through all possible subsets of $\{k+1,\dots,2k\}$.

The right-hand side of \eqref{EXPR:combinatorial-symmetry-trick-factor} equals $(k + \tau_{2k}(n_{2k}))^{2k}\le (k+1)^{2k} \tau_{2k}(n_{2k})^{2k}$,
so upon summing over $(n_1,\dots,n_k;m_1,\dots,m_k;n_{2k})\in D^\star_k(I_1,I_2)$, we conclude that
\begin{equation}
\label{INEQ:reduce-to-divisor-type-problem}
\#N^\star_k(I_1,I_2)
\le (k+1)^{2k} \sum_{(n_1,\dots,n_k;m_1,\dots,m_k;n_{2k})\in D^\star_k(I_1,I_2)} \tau_{2k}(n_{2k})^{2k}.
\end{equation}

We use \eqref{INEQ:reduce-to-divisor-type-problem} to bound $\E_{x_2} \#N^\star_k(I_1,I_2)$.
Note that if $(n_1,\dots,n_k;m_1,\dots,m_k;n_{2k})\in D^\star_k(I_1,I_2)$ and $y:=n_{2k}$ (so that in particular, $m_i-x_2\in [H]$ and $n_i-y\in [-H,H]\setminus \{0\}$ for all $i\in [k]$),
then $y\in (x_1,x_1+H]$ and
\[ (x_2,m_1-x_2,\dots,m_k-x_2,n_1-y,\dots,n_k-y)\in B_k(X,H;y), \]
in the notation of Lemma~\ref{Lem: yth}.
Therefore, summing \eqref{INEQ:reduce-to-divisor-type-problem} over $x_2\in [X]$ gives the inequality

\begin{equation*}
X\cdot \E_{x_2} \#N^\star_k(I_1,I_2)
= \sum_{x_2\in [X]} \#N^\star_k(I_1,I_2)
\ll_k \sum_{y\in (x_1,x_1+H]} \tau_{2k}(y)^{2k} \cdot |B_k(X,H;y)|.
\end{equation*}
 We next apply Lemma~\ref{Lem: yth} to give an upper bound on $|B_k(X,H;y)|$, which leads to 
\[X\cdot \E_{x_2} \#N^\star_k(I_1,I_2) \ll_k \sum_{y\in (x_1,x_1+H]} \tau_{2k}(y)^{2k}  O(H)^{2k} \cdot \tau_2(y)\tau_k(y)^2\cdot O(1 + X/\rad_k(y)).  \]
Average over $x_1$ by using Lemma~\ref{Lem: average yx}, to get
\begin{equation}
\label{EQN:N'-average-estimate}
\E_{x_1,x_2} \#N^\star_k(I_1,I_2)
\ll_k
O(H)^{2k}\cdot H\cdot X^{-1} (\log X)^{\mathcal{C}_k}.
\end{equation}
This is $\ll_k H^{2k} (\log X)^{-\mathcal{C}_k}$ in our range of $H$. 
By \eqref{INEQ:trivial-type-solution-estimate} and \eqref{EQN:N'-average-estimate}, quantity~\eqref{EQN:recursive-estimate} is $k!^2H^{2k} + O_k(H^{2k-1}) + O_k(H^{2k} (\log X)^{-\mathcal{C}_k})$.
Lemma~\ref{LEM:avg-cross-count} follows.
\end{proof}

\begin{proof}
[Proof of Theorem~\ref{thm: Steinhaus}]
Assume $A$ is large and $H\ll X(\log X)^{-C_k}$, where $C_k = A k^{A k^{A k}}$.
Let $C := 10$, so that the quantity $E(k)=2k^2+2$ in Theorem~\ref{thm: high moments} satisfies
\begin{equation}
\label{EQN:convenient-choice-of-C}
E(k)\le 4Ck^2,
\quad E(k+\ell)\le 5Ck^2\text{ for all $1\le \ell\le k-1$}.
\end{equation}
(This is just for uniform notational convenience.)

\noindent
(a). We prove \eqref{eqn: 2k},
a bound on the quantity
\begin{equation}\label{eqn: L2}
    \E_f \bigl( \E_x |A_H(f, x)|^{2k} -k!\bigr)^{\!2},
\end{equation}
where $A_H(f,x)$ is defined as in \eqref{EQN:notation-A_H(f,x)-for-normalized-partial-sum}.
By expanding the square, we can rewrite \eqref{eqn: L2} as
\begin{equation}\label{eqn: 2}
 \E_f \bigl( \E_x |A_H(f, x)|^{2k} \bigr)^{\!2} -2 k! \E_f \E_x |A_H(f, x)|^{2k}  + k!^{2}.    
\end{equation}

The subtracted term in \eqref{eqn: 2} can be computed by switching the summation: it equals
\begin{equation}\label{eqn: 1}
  -2 k! \E_x  \E_f |A_H(f, x)|^{2k}.
 \end{equation}
We estimate \eqref{eqn: 1} by a combination of trivial bounds (based on the divisor bound \eqref{INEQ:k-fold-divisor-bound}) and the moment estimate in Theorem~\ref{thm: high moments}.
We split the sum $\E_x \E_f |A_H(f, x)|^{2k}$ into two ranges, and apply Theorem~\ref{thm: high moments} and \eqref{EQN:convenient-choice-of-C}, to get that
$X\cdot \E_x \E_f |A_H(f, x)|^{2k}$ equals
\begin{equation*}
\begin{split}
&\sum_{1  \le x \le  H(\log X)^{5Ck^{2}} } \E_f |A_H(f, x)|^{2k}
+ \sum_{ H(\log X)^{5Ck^{2}}  \le x \le X} \E_f |A_H(f, x)|^{2k} \\
&= \sum_{1\le x \le H(\log X)^{5Ck^{2}}} O\Bigl((\log X)^{4 Ck^{2}}\Bigr)
+ \sum_{H(\log X)^{5Ck^{2}}  \le x \le X} \Bigl(k! + O\Bigl((\log X)^{-Ck^{2}}\Bigr)\Bigr).
\end{split}
\end{equation*}
Upon summing over both ranges of $x$ above, it follows that $\E_x \E_f |A_H(f, x)|^{2k} = k! + o_{X\to +\infty}(1)$ in the given range of $H$
(provided $A$ is large enough that $C_k\ge 10Ck^{2}$).

We next focus on the first term in \eqref{eqn: 2}.
We expand out the expression and switch the expectations to get that the first term in \eqref{eqn: 2} is
\begin{equation}\label{eqn: cross term}
    \E_{x_1}\E_{x_2} \E_f |A_H(f, x_1)|^{2k}|A_H(f, x_2)|^{2k}.
\end{equation}
Now we use orthogonality and apply Lemma~\ref{LEM:avg-cross-count} to see that \eqref{eqn: cross term}
is $k!^{2}+o_{X\to +\infty}(1)$ in the given range of $H$
(if $A$ is sufficiently large).
Combining the above together,  \eqref{eqn: 2k} follows. 

\noindent
(b). We prove \eqref{eqn: kl}, a bound on the quantity (in the notation $A_H(f,x)$ from \eqref{EQN:notation-A_H(f,x)-for-normalized-partial-sum})
\begin{equation}\label{eqn: odd-moment}
\E_f \Bigl|\E_x \Bigl[A_H(f,x)^k\, \overline{A_H(f,x)^\ell}\Bigr]\Bigr|^{2}
= X^{-2} \sum_{x_1,x_2\in [X]} \mathcal{B}_H(x_1,x_2),
\end{equation}
where $1\le \ell \le k-1$ and $\mathcal{B}_H(x_1,x_2) :=
\E_f A_H(f,x_1)^k \overline{A_H(f,x_1)^\ell} \overline{A_H(f,x_2)^k} A_H(f,x_2)^\ell$.
This is the same as counting solutions to 
\begin{equation}\label{eqn: product}
  n_1n_2\cdots n_k\cdot m_1m_2\cdots m_{\ell} = n_{k+1}n_{k+2}\cdots n_{k+\ell}\cdot m_{\ell+1}m_{\ell+2}\cdots m_{\ell+k},  
\end{equation}
where $x_1\le n_i\le x_1 + H$ and $x_2 \le m_i \le x_2 +H$ for all $1\le i \le k+\ell$. Suppose that $x_1\ge x_2$. 
The left hand side in \eqref{eqn: product} is  
\[ n_1n_2\cdots n_k\cdot m_1m_2\cdots m_{\ell} \ge x_1^{k}x_2^{\ell},\] while the right hand side in \eqref{eqn: product} is 
\[n_{k+1}n_{k+2}\cdots n_{k+\ell}\cdot m_{\ell+1}m_{\ell+2}\cdots m_{\ell+k} \le (x_1+H)^{\ell} (x_2+H)^{k} \le x_1^{\ell}x_2^{k} (1+ \frac{H}{x_2})^{k+\ell}. \]
To make them equal, we must have 
\[x_1/x_2 \le (x_1/x_2)^{k-\ell} \le (1+\frac{H}{x_2})^{2k}, \]
which implies that (under the assumption $Hk = o(x_2)$)
\[x_2 \le x_1 \le x_2 + O(kH). \]
From now on, we only need to consider two cases:
\begin{enumerate}
    \item $\min(x_1, x_2) \ll kH$;
    \item $|x_1-x_2| = O(kH)$.
\end{enumerate}
We first deal with case (1): $\min(x_1, x_2)\ll kH$.
By the Cauchy-Schwarz inequality,
\[ |\mathcal{B}_H(x_1,x_2)|^2
\ll_{k} (\E_f |A_H(f,x_1)|^{2(k+\ell)}) \cdot (\E_f |A_H(f,x_2)|^{2(k+\ell)}). \]
Theorem~\ref{thm: high moments} and \eqref{EQN:convenient-choice-of-C} imply that $\mathcal{B}_H(x_1,x_2) \ll_k (\log X)^{5Ck^2}$.
So the contribution to \eqref{eqn: odd-moment} over all pairs $(x_1, x_2)$ with
$\min\{x_1, x_2\}\le H$ is at most $\ll 1/(\log X)^{C_k - 10Ck^{2}}$, which is $o_{X\to +\infty}(1)$ by our choice of $C_k$. 

We next deal with case (2): $|x_1-x_2|=O(kH)$. Assume $x_2<x_1$. Then all the variables $m_i, n_j$ are in $[x_2, x_1+H]$, so by Theorem~\ref{thm: high moments} and \eqref{EQN:convenient-choice-of-C}, the the contribution in this case to \eqref{eqn: product} over $x_1,x_2$ is at most
\[ \ll_k XH(\log X)^{10 Ck^{2}} \cdot H^{k+\ell} (\log X)^{ 5Ck^{2}}
\ll X^{2}(\log X)^{15Ck^{2}-C_k} \cdot H^{k+\ell}
= X^{2}\cdot o_{X\to +\infty}(H^{k+\ell}), \]
by our choice of $C_k$. After dividing by $X^2 H^{k+\ell}$, we see that the total contribution to \eqref{eqn: odd-moment} in this case is $o_{X\to +\infty}(1)$.


Combining the two cases above, we obtain the desired \eqref{eqn: kl}.
\end{proof}

\section{Concluding remarks}\label{sec: end}

Recall the exponent $E'(k)$ defined after Theorem~\ref{thm: high moments}.
As we mentioned before, Theorem~\ref{thm: high moments} implies $E'(k)\le E(k)=2k^2+2$, and the truth may be that $E'(k)$ grows linearly in $k$.
The method used in \cite{Breteche,HarperHelson,HeapLind} may help to extend Theorem~\ref{thm: high moments}, i.e.~to improve on the bound $E'(k)\le E(k)$.
Alternatively, one might try to improve on Theorem~\ref{thm: high moments} via Hooley's $\Delta$-function technique from \cite{HooleyDelta}; note that $(x,x+H]\subseteq (x, ex]$ if $H\le x$.

The true threshold in the problem studied in Theorem~\ref{thm: main} is more delicate. A closely related problem is to understand for what range of $H$, as $X\to +\infty$,  the following holds:
\begin{equation}\label{eqn: CLT}
   \frac{1}{\sqrt{H}} \sum_{X<n\le X+H}f(n) \xrightarrow{d} \mathcal{CN}(0,1), 
\end{equation}
where $f$ is a Steinhaus random multiplicative function over the short interval $(X, X+H]$.  In contrast to the problem we studied in this paper, $X$ is first fixed in \eqref{eqn: CLT} and the random multiplicative function $f$ varies. For this question, it is known that \cite{SoundXu} if $H\to +\infty$ and $H\ll X/(\log X)^{2\log 2 -1 +\ee}$, then such a central limit theorem holds.
In the other direction, by using Harper's remarkable results and methods in \cite{HarperLow} one may be able to show that  
\begin{equation}\label{eqn: short}
   \E_f \biggl|\frac{1}{\sqrt{H}} \sum_{X<n \le X+H} f(n)\biggr| = o_{X \to +\infty}(1), ~\text{if}~ H\gg \frac{X}{\exp((\log \log X)^{1/2-\ee})}. 
\end{equation}
(see \cite{SoundXu} for more discussions). Thus, in the above range of $H$, the $\sqrt{H}$-normalized partial sums do not have Gaussian limiting distribution. It would be interesting to know if another choice of normalization would lead to a Gaussian distribution. Now we return to the question we studied in Theorem~\ref{thm: main}. We established ``typical Gaussian behavior'' over a range of the form $H\ll \frac{X}{(\log X)^{W(X)}} = \frac{X}{\exp(W(X)\log\log{X})}$ (where $H\to +\infty$). It seems that to extend the range of $H$ so that such a Gaussian behavior holds, significant new ideas would be needed. 
It would be interesting to understand the whole story for all ranges of $H$, for both the question studied in Theorem~\ref{thm: main} and that in \eqref{eqn: CLT}.

	\bibliographystyle{abbrv}
	\bibliography{CLT}{}
\end{document}